\documentclass[12pt,reqno]{amsart} 

\usepackage{amscd}
\usepackage{amsfonts,amssymb,latexsym}

\usepackage[all]{xy}
\xyoption{arc}
\CompileMatrices

\newcommand{\SL}{{\mathcal{L}}}
\newcommand{\SM}{{\mathcal{M}}}
\newcommand{\SO}{{\mathcal{O}}}
\newcommand{\SU}{{\mathcal{U}}}
\newcommand{\SP}{{\mathcal{P}}}
\newcommand{\SE}{{\mathcal{E}}}

\newcommand{\PP}{\mathbb{P}}
\newcommand{\ZZ}{\mathbb{Z}}
\newcommand{\CC}{\mathbb{C}}
\newcommand{\QQ}{\mathbb{Q}}
\newcommand{\Pic}{\operatorname{Pic}}
\newcommand{\id}{\operatorname{id}}
\newcommand{\surj}{\twoheadrightarrow}
\newcommand{\too}{\longrightarrow}
\newcommand{\rk}{\operatorname{rk}}
\newcommand{\wt}{\widetilde}

\newcommand{\Det}{\operatorname{Det}}
\newcommand{\Ext}{\operatorname{Ext}}

\newcommand{\ch}{\operatorname{ch}}
\newcommand{\codim}{\operatorname{codim}}

\newtheorem{proposition}{Proposition}[section]
\newtheorem{theorem}[proposition]{Theorem}
\newtheorem{lemma}[proposition]{Lemma}

\theoremstyle{definition}
\newtheorem{definition}[proposition]{Definition}

\numberwithin{equation}{section}

\begin{document}

\title[On vector bundles trivial over Hecke curves]{On vector bundles over
moduli spaces trivial on Hecke curves}

\author[I. Biswas]{Indranil Biswas}
\address{School of Mathematics, Tata Institute of Fundamental
Research, Homi Bhabha Road, Mumbai 400005, India}
\email{indranil@math.tifr.res.in}

\author[T. L. G\'omez]{Tom\'as L. G\'omez}
\address{Instituto de Ciencias Matem\'aticas (CSIC-UAM-UC3M-UCM),
Nicol\'as Cabrera 15, Campus Cantoblanco UAM, 28049 Madrid, Spain}
\email{tomas.gomez@icmat.es}



\date{}

\begin{abstract}
Let $M_X(r,\xi)$ be the moduli space of stable vector bundles, on a smooth complex projective 
curve $X$ with $\text{genus}(X)\, \geq\, 3$, of rank $r$ and fixed determinant $\xi$ such that $\deg(\xi)$
is coprime to $r$. 
If $E$ is a vector bundle $M_X(r,\xi)$ whose restriction to every Hecke curve in $M_X(r,\xi)$ 
is trivial, we prove that $E$ is trivial.
\end{abstract}

\maketitle

\section{Introduction}

Moduli spaces of vector bundles on a complex projective curve have a long history. Apart from
algebraic geometry, the context in which these moduli spaces were introduced, they also arise
in symplectic geometry, geometric representation theory, differential geometry and mathematical
physics. Line bundles and higher rank vector bundles on these moduli spaces play central
role in their study. On the other hand, these moduli spaces contain a distinguished class
of rational curves known as Hecke curves. They can be characterized as minimal degree rational
curves passing through a general point on the moduli spaces \cite{Ty} (also proved in
\cite{Su}). These Hecke curves play important role in the geometric
representation theoretic aspect of the moduli spaces and also in the computation of cohomology
of coherent sheaves on the moduli spaces.

Here we study restriction of vector bundles on moduli spaces to the Hecke curves. To describe
the result proved here, fix a smooth complex projective curve $X$ of genus at least three.
Let $\xi$ be a line bundle on $X$ with genus $g\,\geq\, 3$ and $r\, \geq\, 2$ an integer coprime to $\deg (\xi)$.
Let $M_X(r,\xi)$ be the moduli space of stable vector bundles on $X$ of rank $r$ and determinant
$\xi$. We prove the following (see Theorem \ref{thm1}):

\begin{theorem}\label{thm0}
Let $E$ be a vector bundle on $M_X(r,\xi)$ such that its restriction to every Hecke curve on 
$M_X(r,\xi)$ is trivial. Then $E$ is trivial.
\end{theorem}

Our motivation to study this problem comes from the result that says that a
vector bundle on a projective space $\PP^N$ is trivial when the
restriction of it to every line is trivial (in fact, it is enough to check
it for lines through a fixed point, cf. \cite[p. 51, Theorem
3.2.1]{OSS}). In this article we are replacing $\PP^N$ by
$M_X(r,\xi)$, and lines in $\PP^N$ by Hecke curves in $M_X(r,\xi)$
(which are also rational curves of minimal degree). 

To prove Theorem \ref{thm0} we crucially use a theorem of Simpson which says that a semistable vector
bundle $W$ on a smooth complex projective variety admits a flat holomorphic connection if
$c_1(W)\,=\, 0\,=\, c_2(W)$. It is relatively straightforward to deduce that the vector bundle $E$ in
Theorem \ref{thm0} is semistable and $c_1(E)\,=\, 0$. Almost all of our work is devoted in proving
that $c_2(E)\,=\, 0$.

\section{Cohomology of moduli space}

Let $X$ be a smooth complex projective curve of genus $g$. The results
of this sections hold for $g\, \geq \, 2$.
Fix an integer $r\, \geq\, 2$ and also fix a line bundle $\xi$ on $X$ such that
$\deg(\xi)$ is coprime to $r$. Let $$M\,=\,M_X(r,\xi)$$ be
the moduli space of stable bundles on $X$ of rank $r$ and degree
$\deg(\xi)$. This moduli space $M$ is a smooth projective variety of
dimension $(r^2-1)(g-1)$. There is
a Poincar\'e bundle $\SP$ on $X\times M$; two different Poincar\'e bundles
on $X\times M$ differ by tensoring with a line bundle pulled back from $M$.
It is known that
\begin{equation}\label{pm}
\Pic(M)\,=\,\ZZ
\end{equation}
\cite[p.~69]{Ra}, \cite[p.~78, Proposition 3.4(ii)]{Ra}. The
ample generator of $\Pic(M)$
will be denoted by $\SO_M(1)$. The degree of any torsionfree coherent
sheaf $F$ on $M$ is defined to be
$$
\deg(F)\, :=\, \big(c_1(F)\cup c_1(\SO_M(1))^{(r^2-1)(g-1)-1}\big)\cap [M]\, \in\, \mathbb Z\, .
$$

Let $U$ be a rank $r$ vector bundle on $X\times T$ such that for every point
$t\, \in\, T$, the restriction
$U_t\,:=\, U|_{X\times t}$ is stable and has determinant
$\xi$. Let $$\phi\,:\,T\,\too\, M\,=\, M_X(r,\xi)$$ be the corresponding classifying morphism.
Define
$$
\Det U \,:=\, \big( \det(Rp^{}_T{}^{}_* U) \big)^{-1}\,:=\,
\big( \det(R^0p^{}_T{}^{}_* U) \big)^{-1}\otimes \big( \det(R^1p^{}_T{}^{}_* U) \big)^{-1}
\, \longrightarrow\, T\, ,
$$
where $p_T\, :\, X\times T\, \longrightarrow\, T$ is the natural
projection. In \cite[Proposition 2.1]{Na1} there is a formula to
calculate the degree of the classifying morphism $\phi$, which in our
case ($r$ and $d$ are coprime) becomes
\begin{equation}\label{eq:pullbacko1}
\phi^* \SO_M(1) \,=\, (\Det U)^r \otimes (\bigwedge\nolimits^r U_{p})^{d+r(1-g)}
\end{equation}
where $p\,\in\, X$ is any point, and $U_p\,=\,U\vert_{p\times M}$.
Applying this to the Poincar\'e bundle $U\,=\, \SP$, it follows that
\begin{equation}\label{eq:degmod}
(\deg \SP_p) \cdot d \,\equiv\, 1 \ \ \mod\ r
\end{equation}
(see \cite[p.~75, Remark 2.9]{Ra} and \cite[p.~75, Definition 2.10]{Ra}).
Using the slant product operation, construct the integral cohomology classes
\begin{equation}\label{cg}
f_2\,:=\,c_2(\SP)/[X]\, \in\, H^2(M,\, {\mathbb Z})\, ,\ \ a_2\,:=\,c_2(\SP)/[p]
\, \in\, H^4(M,\, {\mathbb Z})
\end{equation}
$$
\text{ and } \ \
f_3\,:=\, c_3(\SP)/[X]\, \in\, H^4(M,\, {\mathbb Z})\, ,
$$
where $[X]\,\in \, H_2(X,\, {\mathbb Z})$ and $[p]\,\in\, H_0(X, \,{\mathbb Z})$ are the 
positive generators.

The following result is standard.

\begin{proposition}\label{prop1}\mbox{}
\begin{itemize}
\item The integral cohomology of $M$ has no torsion.

\item The rank of $H^2(M,\, {\mathbb Z})$ is 1. The cohomology class $f_2$ in \eqref{cg}
generates the $\mathbb Q$--vector space $H^2(M,\, {\mathbb Q})$. 

\item For $r\, \geq\, 3$, the rank of $H^4(M, \, {\mathbb Z})$ is $3$, while
${\rm rank}(H^4(M,\, {\mathbb Z})) \,=\, 2$ for $r\,=\,2$.
The $\mathbb Q$--vector space $H^4(M,\, {\mathbb Q})$ is generated by
\begin{equation}\label{eq:generators}
(f_2)^2\, , \ \ a_2\ \ \text{ and} \ \ f_3\, ,
\end{equation}
where $a_2$ and $f_3$ are defined in \eqref{cg}.
(Note that $f_3\,=\, 0$ if $r\,=\, 2$.)
\end{itemize}
\end{proposition}

In \cite[p.~578, Theorem 9.9]{AB} it is proved that $H^*(M,\, {\mathbb Z})$ is torsionfree. See 
\cite[p.~582, Proposition 9.13]{AB} for the second statement. For the third statement, see 
\cite[p.~543, Proposition 2.20]{AB}, \cite[p.~114, Section 2]{JK}, \cite[p.~2, Theorem 1.5]{BR}.

The cohomology class $a_2$ in Proposition \ref{prop1}(3) depends 
on the choice of Poincar\'e bundle $\SP$. In the following lemma we show that $c_2(\SP)/[p]$ in
\eqref{cg} can be 
replaced by $c_2({\rm End}(\SP))/[p]$, which does not depend on the choice of Poincar\'e 
bundle. This would simplify our later calculations.

\begin{lemma}\label{newgenerators}
The cohomology classes
$$
(f_2)^2,\ \ b_2\,:=\,c_2({\rm End}(\SP))/[p] \ \ \text{ and} \ \ \quad f_3
$$
generate $H^4(M, \, {\mathbb Q})$.
(Note that $f_3\,=\, 0$ if $r\,=\, 2$.)
\end{lemma}

\begin{proof}
In view of Proposition \ref{prop1}(3), it
suffices to prove that $a_2$ in \eqref{eq:generators}
can be expressed as a function of the classes \eqref{eq:generators}.
The slant product
$$
H^k(X\times M,\, {\mathbb Z}) \otimes H_\ell (X,\, {\mathbb Z})
\,\too\, H^{k-\ell}(M,\, {\mathbb Z})\, \ \ \quad (\eta,c)\,\longmapsto\, \eta/c
$$
satisfies the following natural condition \cite[p. 264, (29.23)]{GH}:
For morphisms $f\,:\,X'\,\too\, X$ and $g\,:\,M'\,\too\, M$,
$$
((f \times g)^*\eta)/c \,=\, g^* (\eta/f_*(c))\, .
$$
In particular, if $i\,:\,x\,\hookrightarrow\, X$ is a point 
and $\eta\,\in\, H^k(X\times M,\, {\mathbb Z})$, then
\begin{equation}\label{eq:restriction}
\xi/[i(x)] \,=\, (i\times \id_M)^* \xi/[x]\,=\, \xi|_{i(x)\times M}\ \,\in\, H^k(M,\, {\mathbb Z})\, .
\end{equation}
Now consider $c_2(\SP)/[p]$ in \eqref{cg}. We have
$$
c_2({\rm End}(\SP))/[p] \,=\, 
-2r c_2(\SP)/[p] + \big( c_1(\SP)^2\big) /[p]\, .
$$
Using \eqref{eq:restriction} it follows that
$\big( c_1(\SP)^2\big) /[p]\,=\,c_1(P_p)^2$. Note that
Proposition \ref{prop1}(2) says that $c_1(P_p)\,=\,kf_2$ for some $k\,\in\, \QQ$.
Consequently, we have
$$
a_2\,=\, \frac{-b_2+k^2 (f_2)^2}{2r}\, ,
$$
which proves the lemma.
\end{proof}

\section{Hecke curves}

In this section we recall the definition and basic properties of Hecke
curves. 
Let $x\in X$ be a point. Let $F$ be a vector bundle on $X$ with 
determinant $\xi(x)=\xi\otimes \SO_X(x)$. Let $q:F\too k_x$ be a
surjection to the sky-scraper sheaf $k_x$. We obtain a short exact
sequence
\begin{equation}\label{eq:hecke}
0\,\too\, E \,\too\, F \,\too\, k_x \,\too\, 0
\end{equation}
and we say that $E$ is the Hecke transform of $E$, at the point $x$,
with respect to the quotient $q$. If we let $q$ vary among all the
quotients of $F$ with image $k_x$,
we obtain a family of vector bundles of determinant $\xi$ 
parametrized by the projective
space $\PP(F^\vee_x)$ (we take the convention that $\PP$ is the
projective space of lines)
\begin{equation}\label{eq:family}
0\,\too\, \SE\,\too\, p^*_X F\,\too\, \SO_{\{x\}\times \PP(F^*_x)}(1)\,\too\, 0
\end{equation}

\begin{definition}[{\cite[p.~306, Definition 5.1 and Remark 5.2]{NR}}]
Let $l$, $m$ be integers. 
A vector bundle
$F$ over $X$ is $(l,\,m)$--{\em stable} if, for every
proper subbundle $G$ of $F$ of positive rank,
$$
\frac{\deg (G)+l}{\rk G}\,<\,\frac{\deg (F)+l-m}{\rk F}\, .
$$
\end{definition}
We remark that $(l,\,m)$--stability is a Zariski open condition \cite[Proposition 5.3]{NR}.
Note that $(0,\,0)$--stability is the same thing as usual stability.
If $F$ is $(l,\,m)$--stable in \eqref{eq:hecke}, then $E$ is
$(l,\,m-1)$--stable \cite[Lemma 5.5]{NR}. 
Indeed, if a subbundle $G$ of $E$ contradicts
$(l,\,m-1)$--stability, then the subbundle of $F$ generated by $G$ will
contradict the $(l,\,m)$--stability of $F$.

Let $F$ be a $(0,\,1)$--stable bundle of determinant $\xi(x)$. The family of Hecke
transforms \eqref{eq:family} gives a morphism, which we call Hecke space: 
$$
\psi\,:\,\PP(F_x^\vee) \,\too\, M_X(r,\xi)
$$
This morphism is an embedding \cite[Lemma 5.9]{NR}.
A Hecke curve is the restriction of this morphism to a line 
$\PP^1 \subset \PP(F_x^\vee)$
$$
\psi|_{\PP^1}\,:\,\PP^1\,\too\, M_X(r,\xi)
$$ 
We use the term Hecke ``curve'' instead of ``line'', 
because the degree of this morphism is not one. In fact, using formula
\eqref{eq:pullbacko1} we can calculate $\deg \psi|_{\PP^1}\,:=\,\deg
\psi|_{\PP^1}^* \SO_{M}(1)\,=\,r$.

\begin{lemma}
If $r$ and $d\,=\,\deg \xi$ are coprime, and $g\,\geq\, 3$, 
then the $(0,\,2)$--stable vector bundles with determinant
$\xi(2x)$ give a non-empty Zariski open subset of $M_X(r,\xi(2x))$.
\end{lemma}

\begin{proof}
The case of $(0,\,1)$--stability and $(1,\,0)$--stability is considered in
\cite[Proposition 5.4]{NR}. We
sketch here the analogous argument for $(0,\,2)$--stability.

We shall estimate the dimension of the complement of the subset of
$(0,\,2)$--stable bundles. 
A vector bundle $F$ of degree $d+2$ and rank $n$ is not
$(0,\,2)$--stable if it has a subbundle $F'$
of rank $n'\,<\,n$ and degree $d'$ with $0\,\geq\, ((d+2)-2)r'-d'r\,=\,dr'-d'r$.
Since $r$ and $d$ are coprime, equality cannot hold, so we have
\begin{equation}\label{eq:unstable}
dr'-d'r\, <\, 0\, .
\end{equation}
We have
$$
h\,:=\,\dim H^1(Hom(F/G,\,G))\,=\,(d+2)r'-d'r+r'(r-r')(g-1)
$$
so the subset of $M_X(r,\xi(2x))$ corresponding to bundles which are
not $(0,\,2)$--stable has dimension at most 
$$
\big( r^2(g-1)+1\big) + \big( (r-r')^2(g-1)+1\big) - g + h-1\,=\,
$$
$$
( r^2-rr' +r'{}^2-1))(g-1)+r'(d+2)-d'r
$$
and this is smaller than $\dim M_X(r,\xi(2x))\,=\,(r^2-1)(g-1)$ if and only if 
\begin{equation}\label{eq:ineq}
r'(r-r')(g-1)\,>\,r'(d+2)-d'r\, .
\end{equation}
Now, using \eqref{eq:unstable}, it is easy to see that this holds if $g\geq 3$.
\end{proof}

\begin{lemma}
\label{general}
Let $x\in X$ be a point and let $S\subset M_X(r,\xi)$ be a closed subset of codimension
$\codim(S,\,X)\,\geq\, 2$. Then a general Hecke curve (with respect to $x$) does not intersect $S$.
\end{lemma}

\begin{proof}
This is proved, for instance, in \cite[Lemma 4]{BBGN}. We sketch the
proof for the convenience of the reader.
Let $\SU$ be the universal bundle on $X\times M_X(r,\xi)$, and let 
$$
P_x\,=\,\PP(\SU|_{\{x\}\times M})
$$
A point of $P_x$ over $E$ corresponds to a point in the projective space
$\PP(E_x)\,=\,\PP(\Ext^1(k_x,E))$, so it corresponds to a short exact
sequence as \eqref{eq:hecke}. Let $H_x\,\subset\, P_x$ be the open subset
such that $F$ is $(0,1)$--stable. Then there is a morphism $q\,:\,H_x\,\too
\,M_X(r,\xi(x))$, whose image we call $V$. 
$$
\xymatrix{
{H_x}\ar[r]^-{q}\ar[d]_{p} & {V\subset M_X(r,\xi(x))}\\
{M_X(r,\xi)}
}
$$
The fiber of $p$ is an open subset of $\PP(E_x)$, and the fiber of $q$ is $\PP(F_x^\vee)$. Let $S'
\,=\,p^{-1}(S)$. It is either empty or $\dim S' \,=\,\dim S +r-1$. Consider now the morphism $q$.
The fiber of $S'$ over $F$ is $P(F,x)\cap S$, 
the intersection of $S$ with the Hecke space corresponding to $F$. 
If $q(S')$ has positive codimension, then for generic $F$ the Hecke space does not intersect $S$. On
the other hand, if $q(S')$ is dense in $V$, the generic fiber of the projection from $S'$ to $V$ has dimension
$$
\dim S + r-1-\dim V\,=\,\dim S+r-1-\dim M_X(r,\xi)\,\leq\, r-3\, .
$$
Summing up, for generic $F$, the codimension of $P(F,x)\cap S$ in $P(F,x)$ is at least
two. Then, a general Hecke line defined using $F$, which is by definition a line in $P(F,x)$, does not intersect $S$.
\end{proof}

\section{Hecke transformation on two points of the curve}\label{sec:heckecycle}

Let $F$ be a $(0,\,2)$--stable bundle with rank $r$ and determinant
$\xi(x_1+x_2)\,=\, \xi\otimes{\mathcal O}_X(x_1+x_2)$ 
for fixed points $x_1,\, x_2\,\in\,
X$. We are going to perform Hecke transformations on $F$ over these two points
$x_1,\, x_2$. The parameter space will be
$$
\PP_1\times \PP_2\,:=\, \PP(E_{x_1}^\vee) \times \PP(E_{x_2}^\vee)\,\cong\, 
\PP^{r-1}\times \PP^{r-1}$$ 
For a point $x\,\in\, X$ let
$$i_x\,:\,\PP_1\times \PP_2 \,\longrightarrow\, X\times \PP_1\times \PP_2,\, \ \
(y,\, z)\, \longmapsto\, (x,\, y,\, z)$$
be the inclusion map. Let $p_{\PP_1\times \PP_2}\, :\, X\times\PP_1\times \PP_2\, \longrightarrow\,
\PP_1\times \PP_2$ be the natural projection.

For integers $a,\, b$, the line bundle
${\mathcal O}_{\PP(E_{x_1}^\vee)}(a)\boxtimes {\mathcal O}_{\PP(E_{x_2}^\vee)}(b)$
on $\PP(E_{x_1}^\vee) \times \PP(E_{x_2}^\vee)$ will be denoted by 
$\SO(a,\,b)$. Consider the vector bundle $U$ on
$X\times \PP_1\times \PP_2$ defined by the short exact sequence
\begin{equation}\label{fU}
0 \,\too\, U \,\too\, p^*_X F \,\too\,(i_{x^{}_1})_* p^*_{\PP_1\times \PP_2}\SO(1,\,0) \oplus
(i_{x^{}_2})_* p^*_{\PP_1\times \PP_2}\SO(0,\,1) \,\too\, 0\, .
\end{equation}
Using the fact that $F$ is $(0,\,2)$--stable it can be shown that $U$ is a family of stable 
bundles on $X$. Indeed, for $(p_1,\,p_2)\,\in\, \PP_1\times \PP_2$, if a subbundle $G$ of the 
vector bundle $U_{(p_1,p_2)}\, :=\, U\vert_{X\times(p_1,p_2)}$ on $X$ contradicts the 
stability condition, then the subbundle of $F$ generated by $G$ contradicts the 
$(0,\,2)$--stability of $F$ (see {\cite[p.~307, Lemma 5.5]{NR}}).

From \eqref{fU} it follows that $$(\bigwedge\nolimits^r U_{(p_1,p_2)})\otimes {\mathcal O}_X(x_1+x_2)\,=\,
\bigwedge\nolimits^r F\,=\, \xi\otimes{\mathcal O}_X(x_1+x_2)\, .$$ This implies that
$\bigwedge^r U_{(p_1,p_2)}\,=\,\xi$. Let
\begin{equation}\label{eq:psi}
\psi\,:\,\PP_1\times \PP_2\,\too\, M \,=\, M_X(r,\xi)
\end{equation}
be the corresponding classifying morphism. 

If the point $p$ in \eqref{eq:pullbacko1} is different from $x_1$ and $x_2$, then
$U_p$ (as in \eqref{eq:pullbacko1}) for the family $U$ in \eqref{fU} is evidently trivial. 
Therefore, from \eqref{eq:pullbacko1} it follows that
\begin{equation}\label{rr}
\psi^* \SO_M(1) \,\cong\, \SO(r,\,r)\, .
\end{equation}
We assume that the Poincar\'e bundle is normalized by imposing the condition
(see \eqref{eq:degmod})
\begin{equation}\label{nc}
0\,<\, d'\,:=\, \deg (\SP_p) \,< \,r\, .
\end{equation}
By the universal property of the Poincar\'e bundle, there exist integers
$a_1$, $a_2$ such that
\begin{equation}\label{aa}
(\id_X\times \psi)^* \SP\,=\,U \otimes p^*_{\PP_1\times \PP_2}\SO(a_1,\,a_2)\, .
\end{equation}
Once we restrict the isomorphism in \eqref{aa} to $p\times M$, 
it follows from \eqref{rr} and \eqref{nc} that
$a_1\,=\,a_2\,=\, d'$. Hence, denoting 
$L\,:=\, p^*_{\PP_1\times \PP_2}\SO(d',\,d')$,
$$
(\id_X\times \phi)^* \SP\,=\, U\otimes L\, .
$$
We now calculate the Chern classes
\begin{align*}
c_1(U \otimes L)&=
d\; P + rd' \;(D_1+D_2)\\
c_2(U \otimes L)&=
(1+(r-1)dd') \; P(D_1+D_2) + d'{}^2\frac{r(r-1)}{2} \; (D_1+D_2)^2\\
c_3(U \otimes L)&= -P(D_1^2+D_2^2)+
\big(d'(r-2)+d'{}^2d\frac{(r-1)(r-2)}{2}\big)\; (D_1+D_2)^2+\\
 & \quad + \big( d'{}^3\frac{r(r-1)(r-2)}{6}\big)\; (D_1+D_2)^3\\
c_2({\rm End}(U)) &= -2r\; P(D_1+D_2)\, ,
\end{align*}
where $D_1\,\in\, H^2(X\times \PP_1\times \PP_2,\, {\mathbb Z})$ (respectively,
$D_2\,\in\, H^2(X\times \PP_1\times \PP_2,\, {\mathbb Z})$) 
is the pullback of the first Chern class $H_1\,\in\, H^2(\PP_1\times \PP_2,\, {\mathbb Z})$
(respectively, $H_2\,\in\, H^2(\PP_1\times \PP_2,\, {\mathbb Q})$) 
of the line bundle $\SO (1,\,0)$ (respectively, $\SO(0,\,1)$),
and $P$ is the pullback of the class of a point in $X$.

We calculate the pullback of the generators using the pullback formula
for the slant product:
$$
\psi^* f_2 \,=\, \psi^* (c_2(\SP)/[X])\,=\, 
((\id_X\times \psi)^* c_2(\SP))/[X]
$$
$$
\ =\, c_2(U\otimes L)/[X]\,=\,(1+dd'(r-1)) (H_1+H_2)\, .
$$
Analogously, we get that
\begin{align}
\psi^* f_2^2&\,=\, (1+dd'(r-1))^2 (H_1+H_2)^2\, , \nonumber \\
\psi^* b_2&\,=\,0\, , \label{classesp}\\
\psi^* f_3&\,=\,-(H_1^2+H_2^2) +
\big(d'(r-2)+d'{}^2d\frac{(r-1)(r-2)}{2}\big) (H_1+H_2)^2\, .\nonumber 
\end{align}

\section{Hecke transformation on moving point}

It this section we shall construct a family of vector bundles
parametrized by Hecke curves with a moving point.

Let $W$ be a $(0,\,1)$--stable bundle on $X$ of rank $r \, \geq\, 2$ with determinant
\begin{equation}\label{ex}
\bigwedge\nolimits^r W\,=\, \xi(x_0)\,=\, \xi\otimes {\mathcal O}_X(x_0)\, ,
\end{equation}
for a fixed point $x_0\,\in\, X$, and let $$W\,\surj\, Q$$ be a rank 2 torsionfree quotient; since 
$\text{rank}(W)\,=\, r\, \geq\, 2$, such a quotient $W$ exists ($W^\vee(m)$ is globally generated for $m\,\gg\, 0$, so 
choosing two linearly independent vectors on a fiber, there are two global sections restricting to those vectors, and 
if $G$ is the vector subbundle generated by those two sections, we get a 2-dimensional torsionfree quotient $W\,\surj 
\,G^\vee(m)$). Let $X_1$ be a copy of $X$, i.e., $X_1$ is a curve with a fixed isomorphism with $X$; the parameter 
space that we are going to construct involves several copies, so we shall employ this notation to distinguish between 
them. The points of $X_1$ will parametrize the points used to perform a Hecke transformation. Consider the projective 
bundle
$$
\xymatrix{
{\PP(Q^\vee)} \ar[d]^{\pi} \\ {X_1} \ar@{=}[r]& {X}
}
$$
A point in $y\,\in\, \PP(Q^\vee)$ over $x_1\,=\,\pi(y)\,\in\,X_1$ 
gives to a 1-dimensional quotient $$W_{x_1}\,\surj\, Q_{x_1}\,\surj\, \CC$$ of the 
fiber over $x_1$, so $\PP(Q^\vee)$ is, in a natural way, the parameter space of a family
of Hecke transformations with respect to a moving point. We shall write this family
explicitly. Consider the Cartesian diagram:
\begin{equation}\label{cd}
\xymatrix{
{P_\Delta} \ar[r]^-{i} \ar[d] &
{X\times \PP(Q^\vee)}\ar[d]^{\id_{X}\times \pi} \\
{\Delta } \ar[r] & {X \times X_1}\\
}
\end{equation}
where the morphism at the bottom is the
diagonal embedding $\Delta\,=\,X\,\too\, X\times X_1$, $t\, \longmapsto\,
(t,\, t)$. Note that $P_\Delta$ is a $\PP^1$--bundle over $\Delta$.
In fact it is canonically identified with $\PP(Q^\vee)$ once we invoke the
natural isomorphism between the diagonal $\Delta$ and $X_1\,=\, X$. From
this identification between $P_\Delta$ and $\PP(Q^\vee)$, let
$\SO_{P_\Delta}(1)\, \longrightarrow\, P_\Delta$ be the line bundle corresponding
to the tautological line bundle $\SO_{\PP(\QQ^\vee)}(1)$.

There is a canonical short exact
sequence of sheaves on $X \times \PP(Q^\vee)$:
\begin{equation}\label{cse}
0 \,\too\, F \,\too\, p^*_X W \,\too\, i_*\SO_{P_\Delta}(1) \,\too\, 0\, ;
\end{equation}
recall from \eqref{cd} that $i$ is the inclusion of $P$ in $X \times \PP(Q^\vee)$; here
we consider $F$ as a family of vector bundles on $X$ parametrized by $\PP(Q^\vee)$.
Using the condition that $W$ is $(0,\, 1)$--stable it can be shown that
the vector bundle $F_y\, :=\, F\vert_{X\times y}$ on $X$ is stable
for every point $y\, \in\, \PP(Q^\vee)$. Indeed, if a subbundle $S$ of $F_y$ contradicts the
stability condition, then the subbundle of $W$ generated by $S$ contradicts that
$(0,\,1)$--stability condition for $W$ \cite[p.~307, Lemma 5.5]{NR}.

We are going to calculate the Chern
character of the vector bundle $F$ in \eqref{cse}. The following notation for the Chow
classes on $X \times \PP(Q^\vee)$ will be used:
\begin{itemize}
\item $P_1$ (respectively, $P$) is the pullback of the class of a
point in $X_1$ (respectively, $X$).

\item ${P}_{01}$ is the pullback of the class of a point in $X \times
X_1$.

\item ${\delta}$ is the pullback of the class of the diagonal in
$X \times X_1$.

\item ${D}$ is the pullback of the divisor $\SO_{\PP(Q^\vee)}(1)$ on
$\PP(Q^\vee)$.
\end{itemize}

We can now calculate:
$$
\ch(W)\,=\,r + \big[(d+1) P \big] \, .
$$
Let $i(P_\Delta)\, \subset\, X\times \PP(Q^\vee)$ be the image of the closed inclusion
$i$ in \eqref{cd}. To identify $\ch(\SO_{i(P_\Delta)})$, we first do the
following calculations on $X \times X_1$:
$$
\ch(\SO_{X\times X_1}(-\Delta))\,=\,1-\Delta+\frac{\Delta^2}{2}\, ,
$$
$$
\ch(\SO_{\Delta})\,=\,\Delta-\frac{\Delta^2}{2}\,=\,\Delta-\frac{(2g(X)-2) p }{2}\, ,
$$
where $p\,\in\, H^4(X\times X_1,\, {\mathbb Z})$ is the class of a point in
$X \times X_1$. It follows that
$$
\ch(\SO_{i(P_\Delta)})\,=\,p^*_{X\times X_1} \ch(\SO_\Delta)\,=\,
\delta-(g(X)-1) P_{01}\, .
$$
Now,
$$
\ch(i_*\SO_{P_\Delta}(1))\,=\,\ch(\SO_{i(P_\Delta)}\otimes \SO_{\PP(Q^\vee)}(1))=
\Big(\delta-(g(X)-1) P_{01}\Big)\Big(1+D+\frac{D^2}{2}\Big)
$$
$$
=\, \delta \, + \, \big[\delta \wt{D} - (g(X)-1){P}_{01}\big] \, + \,
\big[ \frac{\delta D^2}{2} + (g(X)-1) P_{01}D \big]\,.
$$
Finally we obtain the Chern character of $F$:
\begin{equation}\label{sfprime}
\ch(F)\,=\,r \, + \, \big[ (d+1)P_2 - \delta \big] \, + \,
\big[ -\delta D + (g(X)-1)P_{01} \big] \, + \, 
\big[ -\frac{\delta D^2}{2} - (g(X)-1) P_{01}D \big]\, .
\end{equation}

It may be clarified that $F$ in \eqref{cse} is a family of stable vector bundles of degree $d$ parametrized by 
$\PP(Q^\vee)$, but the determinant is not fixed. Indeed, if $y\,\in\, \PP(Q^\vee)$ and 
$x_1\,=\,\pi(y)$, then the determinant of the vector bundle corresponding to the point $y$ is 
$(\bigwedge^r W)\otimes {\mathcal O}_X(-x_1)\,=\, \xi\otimes \SO_X(x_0-x_1)$ (see \eqref{ex}).
In particular, the family $F$ induces a morphism from $\PP(Q^\vee)$ to 
the moduli space $M_X(r,d)$ of stable vector bundles on $X$ of rank $r$ and degree $d$. But we 
want a morphism to the fixed determinant moduli space, so we shall tensor this family with an 
$r$-th root of $\SO_X(x_1-x_0)$. Since $x_1\in X_1$ is a moving point, 
to have a family of $r$-th roots we need to pass to a Galois 
cover of the parameter space $X_1$.

Let 
$$
f\,:\,X_1\,\too \,J(X)\, ,\ \ x_1\, \longmapsto\, {\mathcal O}_X(x_1-x_0)
$$
be the Abel-Jacobi map for $X_1$, where $x_0$ is the point in \eqref{ex}
(recall that $X_1\,=\,X$ is a copy of the
same curve, but we make this distinction in notation because of the
different roles they will play in the construction).
This morphism $f$ corresponds to a family of
line bundles on $X$ of degree zero parametrized by $X_1$, i.e., a line bundle $\SL$
on $X\times X_1$ such that $\SL|_{X\times x_1}\cong
\SO_{X}(x_1-x_0)$. Let $$w_r\,:\,J(X)\,\too\, J(X)\, ,\ \ L\, \longmapsto\, L^{\otimes r}$$ be the morphism
that sends a line bundle to its $r$-th tensor power. Consider the Cartesian diagram
\begin{equation}\label{dl}
\xymatrix{
{T} \ar[r]^-{f_r} \ar[d]_{t} & J(X) \ar[d]^{w_r}\\
{X_1} \ar[r]^-{f} & J(X)
}
\end{equation}
We note that $T$ is a connected Galois covering of $X_1$, because $w_r$ is a Galois covering and
the homomorphism $f_*\, :\, \pi_1(X_1)\, \longrightarrow\, \pi_1(J(X))$ induced by $f$ is surjective.
It is easy to check that, if the morphism $f$ in \eqref{dl} corresponds to a line
bundle $\SL$ on $X\times X_1$, then the morphism $f_r$ corresponds to a
line bundle $\SM$ on $X\times T$ such that 
$\SM^{\otimes r}\cong (\id_X\times t)^*\SL$, where $t$
is the map in \eqref{dl}. In other words, after pulling back from $X_1$ to $T$, the family
$\SL$ admits an $r$-th root namely $\SM$.

Let $Z$ be defined by the Cartesian diagram
\begin{equation}\label{defz}
\xymatrix{
 {Z} \ar[r]^-{q} \ar[d]_{\pi^{}_T} & \PP(Q^\vee) \ar[d]^{\pi}\\
 {T} \ar[r]^{t} &{X_1}\\
}
\end{equation}
Finally, we define the vector bundle
$$
U\,=\,(\id_{X}\times q)^*F\otimes (\id_X\times {\pi}_T{})^*\SM^{-1}
$$
on $X\times Z$, where $F$ is the vector bundle in \eqref{cse}. From the construction of
$U$ it is evident that $U$ is a vector bundle on $X \times T$ which
represents a family of vector bundles on $X$ of fixed determinant $\xi$.
Also, this is a family of stable vector bundles, because $F$ is a family
of stable vector bundles. Consequently, we have a classifying morphism
\begin{equation}\label{eq:varphi}
\varphi\,:\,Z \,\too\, M \, .
\end{equation}

Our objective now is to calculate the class $\varphi^* b_2\,=\,
\varphi^*(c_2({\rm End}(\SP)))/[p]$. Note that the advantage of working
with ${\rm End}(\SP)$ instead of $\SP$ is that we do not have to worry
about normalization of the Poincar\'e bundle, and also the
tensorization by the line bundle $\SM$ will not appear in the calculation. 
We have
$$
{\rm End}(U)\,=\, (\id_{X}\times q)^* {\rm End}(F)\, ,
$$
and
$$
c_2({\rm End}(F))\,=\,-\ch_2({\rm End}(F))-\frac{1}{2} c_1({\rm End}(F))^2 \,=\, -\ch_2({\rm End}
(F))
$$
$$
=\, -[\ch(F^\vee)\otimes \ch(F)]_2\,=\, 2r\ch_2(F)-\ch_1(F)^2
$$
$$
=\, 2r\,\delta D + (2r(g-1)-2(d+1)+(2-2g))\,P_{01}
$$
(see \eqref{sfprime}).
Recall that ``slanting with the class of a point is the same thing as
restriction to the slice'' (formula \eqref{eq:restriction}). 
It follows that, if $p$ is a point in $X$ and 
$[p]\,\in\, H_0(X,\, {\mathbb Z})$ is its homology class, 
then $P_{01}/[p]\,=\,0$ and $\delta
D/[p]\,=\, [\varpi]\,\in\, H^4(\PP(Q^\vee,\, {\mathbb Q})$,
where $[\varpi]\,\in\, H^4(\PP(Q^\vee),\, {\mathbb Z})$ is the positive generator.
So, we have
$$
c_2({\rm End}(F))/[p]\,=\,2r[\varpi] \,\in\, H^4(\PP(Q^\vee),\, {\mathbb Q})\, .
$$
Also, $c_2({\rm End}(F))/[p]\,=\,c_2({\rm End}(F_p))$,
where $F_p$ is the restriction of $F$ to the slice $p\times
\PP(Q^\vee)$ and then 
\begin{equation}\label{classesz}
\varphi^*b_2\,=\,\varphi^*c_2({\rm End}(\SP_p))\,=\,c_2({\rm End}(U_p))\,=\,
q^* c_2({\rm End}(F_p))=\deg(q) 2r \,=\, r^{2g} 2r\, ,
\end{equation}
where $\SP_p$ and $U_p$ respectively are the restrictions of $\SP$
and $U$ to the slice $p\times M$.

\section{Vanishing of Chern classes}

Let $E$ be a vector bundle on $M\,=\,M_X(r,\xi)$ 
such that the restriction of $E$ to every Hecke curve is trivial. 
Throughout this section, $E$ would satisfy this condition.

{}From the above condition it can be deduced that 
\begin{equation}\label{eq:c1}
c_1(E)\,=\, 0\, .
\end{equation}
Indeed, $H^2(M,\, {\mathbb Z})\,=\,\ZZ$, and a Hecke curve $f\,:\,\PP^1 \,\too\, M$
induces an injection on the second cohomology 
$$
f^* \,:\, H^2(M, \,{\mathbb Z}) \,\too\, H^2(\PP, \, {\mathbb Z})\,\cong\, \ZZ\, .
$$
Now, $f^*c_1(E)\, =\, c_1(f^*E)\,=\,0$, and hence \eqref{eq:c1} holds.
Recall that the rank of $H^4(M,\, {\mathbb Z})$ is 3 when $r\,\geq\, 3$, and it is $2$ when
$r\,=\,2$, and the generators are given by Lemma \ref{newgenerators}.

\begin{lemma}\label{lem:vanishpsi}
Let $\psi\,:\,\PP_1\times\PP_2\,\too\, M$ be the morphism \eqref{eq:psi}.
The pullback $E'\,:=\,\psi^*(E)$ is a trivial vector bundle on $\PP_1\times\PP_2$.
\end{lemma}

\begin{proof}
For every line $l_1\,\subset \,\PP_1$ and every point $p_2\,\in\, \PP_2$, the restriction of 
$\psi$ to $l_1\times p_2\,\cong\, \PP^1$ is a Hecke curve, so $E'\vert_{l_1\times p_2}$ is trivial by 
the hypothesis. A vector bundle on a projective space 
is trivial if it is trivial 
when restricted to every line on the projective space
(\cite[p. 51, Theorem 3.2.1]{OSS}). Consequently, $E'\vert_{\PP_1\times p_2}$ is trivial, and 
this is true for every point $p_2\,\in\, \PP_2$.

Therefore $E'$ descends to a vector bundle $F$ on $\PP_2$, i.e., there is a vector bundle $F$ on 
$\PP_2$ such that $q^*F\,\cong\, E'$, where $q$ is the projection of $\PP_1\times\PP_2$
to $\PP_2$. In fact $F\,=\, q_*E'$.

Note that, for any $p_1\,\in\,\PP_1$, the restriction $E'\vert_{p_1 \times \PP_2}$
is isomorphic to $F$. As before, for every line $l_2$ in $\PP_2$, the 
restriction of $\psi$ to $p_1\times l_2$ is a Hecke curve, so $F$ is trivial on $p_1\times l_2$.
Hence $F$ is trivial by the above argument. Consequently, $E'\,=\,q^* F$ is trivial.
\end{proof}

\begin{lemma}
Let $\varphi\,:\,Z\,\too\, M$ be the morphism in \eqref{eq:varphi}.
The pullback $E_Z\,:=\,\varphi^*(E)$ has Chern classes
$$
c_1(E_Z)\,=\,0\,\in\, H^2(Z,\, {\mathbb Z})\ \ { and } \ \ c_2(E_Z)\,=\,0 \,\in\,
H^4(Z,\, {\mathbb Z})\, .
$$
\end{lemma}

\begin{proof}
The vector bundle $E$ on $M$ has $c_1(E)\,=\,0$ (see \eqref{eq:c1}), and hence
$c_1(E_Z)\,=\,0$.

The scheme $Z$ fibers over a curve $\pi_T\,:\,Z\,\too\, T$, and the restriction of
$\varphi$ to any fiber is a Hecke curve (see \eqref{defz}). Hence, the
vector bundle $E_Z$ is trivial on the fibers of $\pi_T$. This implies
that $E_Z$ descends to $T$, i.e., there exists a vector bundle $F$ on $T$
such that $E_Z\,=\,\pi_T^* F$. In fact, $F\,=\, \pi_{T*} E_Z$.

Since $F$ is a vector bundle on a curve, namely $T$, it follows that
$c_2(F)\,=\,0$ (as $H^4(T,\, {\mathbb Z})\,=\, 0$). Therefore, we have
$c_2(E_Z)\,=\,\pi^*_T c_2(F) \,=\,0$.
\end{proof}

\begin{proposition}\label{c2}
The second Chern class $c_2(E)\,\in\, H^4(M, \, {\mathbb Q})$ 
of the vector bundle $E$ on $M$ is zero.
\end{proposition}

\begin{proof}
Using Lemma \ref{newgenerators}, we write $c_2(E)$ as a combination of the
generators with coefficients in $\QQ$:
\begin{equation}\label{eq:linearcombination}
c_2(E)\,=\,\alpha (f_2)^2 + \beta b_2 +\gamma f_3\, .
\end{equation}
Consider the pullback of $c_2(E)$ by the morphism $\psi$ in \eqref{eq:psi}.
We have $\psi^*c_2(E)\,=\, 0$ by Lemma \ref{lem:vanishpsi}, and hence using
\eqref{classesp} it follows that
\begin{align}\label{peq}
\psi^*c_2(E)&\,=\, 0\,=\,\\
 & \Big(\alpha(1+dd'(r-1))^2 + \gamma(d'(r-2)+d'{}^2d\frac{(r-1)(r-2)}{2})\Big) (H_1+H_2)^2 - \gamma(H_1^2+H_2^2)\nonumber
\end{align}
It is easy to check, using \eqref{eq:degmod}, that $1+dd'(r-1)\neq 0$.

If rank $r\,=\, 2$, then recall that $f_3\,=\,0$, we set $\gamma\,=\,0$,
and therefore \eqref{peq} gives $\alpha\,=\,0$.

On the other hand, if $r\,>\,2$, then the classes $H_1^2+H_2^2$ and
$(H_1+H_2)^2$ are linearly independent in $H^4(\PP_1\times\PP_2,\, {\mathbb Z})\,\cong\,
H^4(\PP^{r-1}\times\PP^{r-1},\, {\mathbb Z})\,\cong\, \ZZ^3$, so \eqref{peq} 
implies $\alpha\,=\,\gamma\,=\,0$.

Summing up, for any rank $r\,\geq\, 2$ we have
\begin{equation}\label{peq2}
c_2(E)\,=\,\beta b_2\, .
\end{equation}
Pulling back \eqref{peq2} by the map $\varphi$ in \eqref{eq:varphi}, and
using \eqref{classesz} we get that
$$
0 \,= \, \beta r^{2g} 2r\, .
$$
Therefore, $\beta\,=\,0$, and hence $c_2(E)\,=\,0$ by \eqref{peq2}.
\end{proof}

\begin{theorem}\label{thm1}
Let $E$ be a vector bundle on $M\,=\,M_X(r,\xi)$ satisfying the condition that
the restriction of $E$ to every Hecke curve is trivial. Then $E$ is trivial.
\end{theorem}

\begin{proof}
The restriction of $E$ to a Hecke curve on $M$ is trivial. From this
it can be deduced that $E$ is semistable. Indeed, if $E$ is not semistable, there is a
coherent subsheaf
$$
V\, \subset\, E
$$
such that $E/V$ is torsionfree, and
\begin{equation}\label{de}
\frac{\deg (V)}{{\rm rank}(V)} \, >\, \frac{\deg (E)}{{\rm rank}(E)}\,=\, 0
\end{equation}
(see \eqref{eq:c1}). Now, for a general Hecke curve ${\mathbb P}\, \subset\, M$, the
restriction $V\vert_{\mathbb P}$ is torsionfree ($V$ fails to be
locally free in codimension at least 2, so we can apply Lemma \ref{general}). Moreover, for any smooth closed curve $C\,
\subset\, M$, and any torsionfree coherent sheaf $W$ on $M$ which is locally free on $C$,
we have
\begin{equation}\label{dp}
\deg (W\vert_C)\, =\, \deg (C)\cdot \deg (W)\, .
\end{equation}
Indeed, this follows from the facts that
\begin{itemize}
\item $\text{Pic}(M)\,=\, \mathbb Z$ (see \eqref{pm}), and

\item $\deg (C)$ and $\deg (W)$ are computed using the same ample line bundle bundle, namely
the ample generator of $\text{Pic}(M)$.
\end{itemize}
Also, we have $\deg (C) \, >\, 0$ because the restriction, to $C$, of the ample generator of $\text{Pic}(M)$
is ample. Now, from \eqref{de} and \eqref{dp} it is deduced that
$$
\deg (V\vert_{\mathbb P}) \, >\, 0
$$
for any Hecke curve ${\mathbb P}\, \subset\, M$ that is contained in the open subset where
$V$ is locally free. But the trivial bundle $E\vert_{\mathbb P}$ on $\mathbb P$ does not
contain any subsheaf of
positive degree. From this contradiction we conclude that $E$ is semistable.

Since
\begin{itemize}
\item $E$ is semistable,

\item $c_1(E)\,=\, 0$ \eqref{eq:c1} and

\item $c_2(E)\,=\, 0$ (Proposition \ref{c2}),
\end{itemize}
the vector bundle $E$ admits a filtration of subbundles
$$
0\,=\, E_0\, \subset\, E_1\, \subset\, \cdots\, \subset\, E_{\ell-1}\, \subset\, E_\ell\,=\, E
$$
such that for every $1\, \leq\, i\, \leq\, \ell$, the quotient $E_i/E_{i-1}$ is a stable vector
bundle with $c_1(E_i/E_{i-1})\,=\, 0\, =\, c_2(E_i/E_{i-1})$ \cite[p.~39, Theorem 2]{Si}
(note that $\ell\,=\, 1$ is allowed). This implies that $E$ admits a flat holomorphic connection
\cite[p.~40, Corollary 3.10]{Si} (set the Higgs field to be zero in \cite[Corollary 3.10]{Si});
see \cite[p.~4015, Proposition 3.10]{BS} for an extension of this result.

Since $E$ admits a flat holomorphic connection it is given by a representation of
$\pi_1(M)$ in $\text{GL}(r,{\mathbb C})$, where $r$ is the rank of
$E$. On the other hand, $M$ is simply connected \cite[p.~581, Theorem 9.12]{AB}.
Therefore, the vector bundle $E$ is trivial.
\end{proof}

As we explained in the introduction, this result is analogous to the
result that says that a vector bundle $E$ on $\PP^N$ is trivial when it is
trivial when restricted to every line passing through some fixed point
$p\in \PP^N$ \cite[p. 51, Theorem
3.2.1]{OSS}. We will finish with some remarks, comparing our result
to the classical one.

The idea of the proof of in \cite[p. 51, Theorem
3.2.1]{OSS} is the
following. Let $\sigma\,:\,\wt \PP^N\,\too\, \PP^N$ be the blowup of $\PP^N$
at $p$. Note that $E\,=\,E\otimes \sigma_*\SO_{\wt\PP^N}\,=\,\sigma_*\sigma^* E$, 
so it is enough to show that
the pullback $\sigma^* E$ is trivial.
By hypothesis, the pullback $\sigma^* E$ is trivial on the
fibers of the projection of $\wt \PP^N$ to the exceptional divisor
$G$, so it descends to a vector bundle $F$ on $G$. But the restriction
of the blowup $\sigma$ to the exceptional divisor $G$ is constant, so
it follows that $F$ is trivial (with fiber canonically isomorphic to
the fiber of $E$ on $p$). Therefore, the pullback $\sigma^*E$ 
of $E$ to the blowup $\wt\PP^N$ is trivial, and the result is proved.

In this argument, we use the fact that the lines passing through a
point $p$ sweep the whole projective space (in order to prove
$\sigma_* \sigma^*E\,=\,E$).
But the Hecke curves through a fixed point in the moduli space do not
sweep the whole moduli space. This is why we have to use a different
argument to show that $E$ is trivial. 

If we apply our method to the classical setting, we get a slightly
stronger result.

\begin{lemma}\label{strongclassical}
Let $p$ be a point in $\PP^N$, and fix a plane $\PP^2\subset \PP^N$
Suppose that a vector bundle $E$ on 
$\PP^N$ is trivial on every line in $\PP^2$ passing through the point
$p$. Then $E$ is the trivial bundle.
\end{lemma}

\begin{proof}
It clear that $c_1(E)\,=\,0$, because the restriction to a line is
trivial.

Let $\PP^1$ be a line in $\PP^2$ containing $p$. 
Let $S$ be a subsheaf of $E$ of rank $\rk S\,<\,\rk E$. 
If $Z(E)$ is the subset where $E$ is not locally free, then there is an
automorphism $\varphi$ of $\PP^N$ such that $\PP^1$ is disjoint to
$Z(\varphi^*(E))\,=\,f^{-1}(E)$. Since $E$ is unstable if and only if
$\varphi^*(E)$ is unstable, we may assume that $Z(E)$ is disjoint with
$\PP^1$, and then 
$$
\deg(S)\,=\,\deg(S|_{\PP^1})\,\leq\, 0
$$
where the last equality follows from the fact that $E|_{\PP^1}$ is
trivial, and hence semistable. This proves that $E$ is semistable.

We will now show that $c_2(E)\,=\,0$. It is enough to show that
the restriction $E'\,=\,E|_{\PP^2}$ has $c_2(E')\,=\,0$.
Now let $\sigma\,:\,\wt\PP^2\,\too\,\PP^2$ be the blowup of $\PP^2$ at a
point. To show that $c_2(E')\,=\,0$, it is enough to show $c_2(\sigma^*
E')\,=\,0$.
By hypothesis, the restriction of $\sigma^* E'$ to the fiber of
the projection of $\wt\PP^2$ to the exceptional divisor $G$ are
trivial, so $E'$ is the pullback of a vector bundle $F$ on $G$. But
the dimension of $G$ is one, so $c_2(F)\,=\,0$, and also $c_2(E')\,=\,0$.

Summing up, $E$ is semistable, $c_1(E)\,=\,0$ and $c_2(E)\,=\,0$, so apply
Simpson's results as in the proof of Theorem \ref{thm1}.
\end{proof}

In the case of $\PP^N$ it is enough to check with the lines passing
through a fixed point of $\PP^N$, so it is natural to ask if something
similar holds for the moduli space. Our proof would not work if we
only consider Hecke curves through a fixed point: when we prove
semistability of the vector bundle $E$, we consider a potentially
destabilizing subsheaf $V$, and we restrict to a Hecke curve where $V$
is locally free, but $V$ might not be locally free on the chosen point
$p$. In the proof of Lemma \ref{strongclassical} this is solved by
using the large automorphism group of $\PP^N$, but this trick does not
seem possible in the case of the moduli space, whose automorphism is
very small. Another point where our proof would fail is Lemma
\ref{lem:vanishpsi} we use Hecke curves which do not intersect.

\section*{Acknowledgements}

The comments of the referee led us to a gap in a previous version, which has
been corrected. We are very grateful to the referee for this.
The second author thanks Roberto Mu\~noz for discussions. 
Part of this work was done during visits to the
Kerala School of Mathematics, Tata Institute of Fundamental
Research and to the
International Center for Theoretical Sciences (ICTS, during the
program Moduli of Bundles and Related Structures, code: 
ICTS/mbrs2020/02).

The first-named author is partially supported by a J. C. Bose Fellowship, and
school of mathematics, TIFR, is supported by 12-R$\&$D-TFR-5.01-0500.
The second-named author is supported by
Ministerio de Ciencia e Innovaci\'on of Spain
(grant MTM2016-79400-P, and 
ICMAT Severo Ochoa project SEV-2015-0554).


\end{document}